\documentclass[preprint,11pt]{article}
\usepackage{amssymb,amsfonts,amsmath,amsthm,amscd,bbold,dsfont,mathrsfs}
\DeclareMathAlphabet{\mathpzc}{OT1}{pzc}{m}{it}

\newtheorem{propo}{Proposition}[section]
\newtheorem{lemma}[propo]{Lemma}
\newtheorem{definition}[propo]{Definition}

\newtheorem{thm}[propo]{Theorem}

\newtheorem{prop}[propo]{Proposition}

\def\oS{S^{\rm c}}

\newcommand{\reals}{{\mathds R}}
\newcommand{\integers}{{\mathds Z}}
\newcommand{\naturals}{{\mathds N}}
\newcommand{\eqnsection}{\renewcommand{\theequation}{\thesection.\arabic{equation}}
      \makeatletter \csname @addtoreset\endcsname{equation}{section}\makeatother}
\def\eps{\epsilon}

\def\l|{\left|\left|}
\def\r|{\right|\right|}

\def\E{\mathds E}
\def\1{\mathds 1}
\def\prob{{\mathds P}}

\def\ind{{\mathds I}}

\def\ve{\varepsilon}

\def\de{{\rm d}}
\def\reals{{\mathds R}}

\def\ux{\underline{x}}

\def\M{{\mathpzc M}}
\def\Tree{{\sf T}}

\def\Ball{{\sf B}}
\def\Ell{{\sf E}}
\def\cEll{{\sf E}^{\rm c}}

\def\cBall{{\sf B}^{\rm c}}

\def\sTV{\mbox{\tiny\rm TV}}

\def\G{\mathcal{G}}
\def\H{\mathcal{H}}
\def\M{\mathcal{M}}
\def\cB{\mathcal{B}}
\def\mm{\mathfrak{m}}

\def\normeq{\cong}

\def\cX{{\cal X}}

\def\F{{\sf F}}

\def\hmu{\hat{\mu}}

\def\root{\o}

\def\cTree{{\sf T^c}}

\def\limMes{{\bar \nu}}
\def\bnu{\bar{\nu}}
\def\Var{{\rm Var}}
\def\Cov{{\rm Cov}}
\def\hf{\hat{f}}
\def\barf{\bar{f}}

\eqnsection
\begin{document}
\title{The weak limit of Ising models on locally tree-like graphs}

\author{Andrea Montanari\thanks{Department of Electrical Engineering and Department of Statistics, Stanford University},
\;\;\;\; Elchanan Mossel\thanks{Faculty of Mathematics and Computer Science, Weizmann Institute and Departments of Statistics and Computer Science, UC Berkeley}\;
\;\;and \;\;
Allan Sly\thanks{Microsoft Research, Redmond, WA}
}

\date{\today}

\maketitle

\begin{abstract}
We consider the Ising model with inverse temperature $\beta$ and without external field on sequences of graphs $G_n$
which converge locally to the $k$-regular tree. We show that for such graphs the Ising measure locally weak converges to the symmetric
mixture of the Ising model with $+$ boundary conditions and the $-$
boundary conditions on the $k$-regular tree with inverse temperature $\beta$.
In the case where the graphs $G_n$ are expanders we derive a more detailed understanding by showing convergence of the Ising measure
condition on positive magnetization (sum of spins) to the $+$ measure on the tree.
\end{abstract}

\section{Introduction}

An \emph{Ising model on the
finite graph $G$} (with vertex set $V$,
and edge set $E$) is defined by the following distribution over
$\ux = \{x_i:\, i\in V\}$, with
$x_i\in\{+1,-1\}$
\begin{eqnarray}
\mu(\ux) =\frac{1}{Z(\beta,B)}\, \exp\Big\{\beta\sum_{(i,j)\in E}x_ix_j
+B\sum_{i\in V}x_i\Big\}\, .\label{eq:IsingModel}
\end{eqnarray}
The model is \emph{ferromagnetic} if $\beta\ge 0$ and, by symmetry,
we can always assume $B\ge 0$.
Here $Z(\beta,B)$ is a normalizing constant (partition function).

The most important feature of the distribution $\mu(\,\cdot\,)$
is the `phase transition' phenomenon. On a variety of large
graphs $G$, for large enough $\beta$ and $B=0$, the measure
decomposes into the convex combination of two well separated
simpler components.
This phenomenon has been studied in detail in the case of grids
\cite{Grids1,Grids2,Grids3,Grids4},
and on the complete graph \cite{NewmanEllis}.
In this paper we consider sequences of regular graphs $G_n = (V_n,E_n)$ with
increasing vertex sets $V_n = [n] = \{1,\dots,n\}$ that converge locally
to trees and prove a local characterization of the corresponding
sequence of measures $\mu_{n}(\,\cdot\,)$, which corresponds to the
phase transition phenomenon.

More precisely,
consider the case in which $G_n$ is a sequence of regular graphs of degree
$k\ge 3$ with diverging girth. The neighborhood of $\Ball_i$
any vertex $i$ in $G_n$ converges to an infinite regular tree of degree $k$.
It is natural to assume that the marginal distribution
$\mu_{n,\Ball_i}(\,\cdot\,)$ converges to the marginal of a neighborhood
of the root for an Ising Gibbs measure on the infinite tree.
For large $\beta$, however, there are uncountably many Gibbs measures on the tree so it is natural to ask which is the limit

A special role is played by the plus/minus boundary conditions
Gibbs measures on the infinite tree, to be denoted,
respectively, by $\nu^+(\,\cdot\,)$ and $\nu^-(\,\cdot\,)$.
It was proved in \cite{DemboMontanariIsing} that, for any
$\beta$, and any $B>0$, $\mu_{n}(\,\cdot\, )$
converges locally to $\nu^+$ as $n\to\infty$ and by symmetry when $B<0$
$\mu_{n}(\,\cdot\, )$
converges locally to $\nu^-$ as $n\to\infty$.

In this paper we cover the remaining (and most interesting) case proving that
\begin{eqnarray}
\mu_{n}(\,\cdot\,)\;\;\;\underset{n}{\longrightarrow}
\;\;\;\frac{1}{2}\,\nu^{+}(\,\cdot\,)\;+\;
\frac{1}{2}\,\nu^{-}(\,\cdot\,)\;\;\; \mbox{ for $B=0$ and
any $\beta\ge 0$\, .} \label{eq:Result1}
\end{eqnarray}
In fact, we prove a sharper result. If $\mu_{n,+}(\,\cdot\,)$
and $\mu_{n,-}(\,\cdot\,)$ denote the Ising measure
(\ref{eq:IsingModel}) conditioned to, respectively,
$\sum_{i\in V} x_i>0$ and $\sum_{i\in V} x_i<0$, then we have
\begin{eqnarray}
\mu_{n,\pm}(\,\cdot\,)\;\;\;\underset{n}{\longrightarrow}
\nu_{\pm}(\,\cdot\,)\;\;\; \mbox{ for $B=0$ and
any $\beta\ge 0$\,,}
\end{eqnarray}
and moreover the convergence above holds for almost all vertices of the graph.
Since $\mu_{n} = \frac{1}{2}\mu_{+,n}+\frac{1}{2}\mu_{-,n}$
(exactly for $n$ odd and approximately for even $n$),
this result implies (\ref{eq:Result1}).
%
%
\section{Definitions and main results}
\label{sec:DefinitionMain}

%
%
\subsection{Locally tree-like graphs}
\label{sec:Local}

We denote by $G_n=(V_n,E_n)$ a graph with vertex set
$V_n\equiv[n]=\{1,\dots,n\}$. The distance $d(i,j)$ between
$i,j\in V_n$ is the length of the shortest path from $i$ to $j$ in
$G_n$. Given a vertex $i\in V_n$, we denote by
$\Ball_i(t)$  the set of vertices whose distance from $i$ is at most
$t$ (and with a slight abuse of notation it will also denote the
subgraph induced by those vertices).
We will let $I$ denote a vertex chosen uniformly from the vertices $V_n$, let $U_n$ denote the measure induced by $I$ and let $J$ denote a uniformly
random neighbor of $I$.

This paper is concerned by sequence of graphs $\{G_n\}_{n\in\naturals}$
of diverging size, that converge locally to $\Tree_k$, the infinite rooted tree of degree $k$.
Let $\Tree_k(t)$ be the subset of  vertices of $\Tree_k$ whose distance
from the root $\root$ is at most $t$ (and, by an abuse of notation, the
induced subgraph). For a rooted  tree $T$, we write $T \simeq \Tree_k(t)$ if there is a graph isomorphism between
$T$ and $\Tree_k(t)$ which maps the root of $T$ to that of $\Tree_k(t)$.  The following definition defines what we mean by convergence in the local weak topology.
\begin{definition}
Consider a sequence of graphs $\{G_n\}_{n\in\naturals}$,  and let
$U_n$ be the law of a uniformly random vertex $I$ in
$V_n$.
We say that $\{G_n\}$ \emph{converges locally} to the
degree-$k$ regular tree $\Tree_k$ if, for any $t$,
\begin{eqnarray}
\lim_{n\to\infty}U_n\{\Ball_I(t) \simeq \Tree_k(t)\} = 1\, .
\end{eqnarray}
\end{definition}

Part of our results hold for sequences of expanders
(more precisely, edge expanders), whose definition
we now recall. For a subset of vertices $S\subset V$,
we will denote by $\partial S$ the subset of edges $(i,j)\in E$
having only one endpoint in $S$.
\begin{definition}
The $k$-regular graph $G=(V,E)$ is a $(\gamma,\lambda)$ (edge) expander
if, for any set of vertices $S\subseteq V$ with $|S|\le n\gamma$,
$|\partial S| \ge \lambda S$.
\end{definition}
%

%
%
\subsection{Local weak convergence}

In analogy with the definition of locally tree-like
graph sequences, we introduce local weak convergence for Ising measures.  This done in two different ways.  First one can look at a random vertex and the random configuration in the neighbourhood of the vertex and examine its limiting measure.  Alternatively, we may choose a random vertex and
consider the marginal distribution of the variables in
a neighborhood under the Ising model. This induces (via the random choice of
the vertex) a distribution over probability measures. We can
therefore ask whether this measure converges to a probability
measure over Gibbs measures.

Recall that an Ising measure $\mu$ on the infinite tree $\Tree_k$ may be either defined as a weak limit
of Gibbs measures on $\Tree_k(t)$ or in terms of the DLR conditions,
see e.g.~\cite{georgii88}.
An Ising model is in particular a probability measure over
$\{-1,+1\}^{\Tree_k}$ endowed with the $\sigma$-algebra generated by
cylindrical sets. We let $\G_k$ denote the space of Ising Gibbs measures on
$\Tree_k$ and let $\H_k$ denote the space of all probability
measures on $\{+1,-1\}^{\Tree_k}$.
We endow both these spaces with the topology of weak convergence.
Since $\{+1,-1\}^{\Tree_k}$ is compact, $\G_k$ and $\H_k$
are also compact in the weak topology by Prohorov's theorem.

We define $\M_k$ (respectively $\M^{\G}_k$) to be the space
of probability measures over $(\H_k,\cB_{\H})$ (resp. $(\G_k,\cB_{\H})$),
with $\cB_{\Omega}$ the Borel $\sigma$-algebra.
Also $\M_k$, $\M^{\G}_k$ are compact in the weak topology.

We will use generically $\mu$ for Ising measures on $G_n$
and $\nu$ for Ising measure on $\Tree_k$.
For a finite subset of
vertices $S\subseteq V_n$, we let $\mu^S$ be the marginal of
$\mu$ on the variables
$x_j$,  $j\in S$.
We the shorthand $\mu^t$ for when $S = \Ball_i(t)$ is
the ball of radius $t$ about $i$ ($i$ should be clear from the context).
For a measure $\nu\in \G_k$ we let $\nu^t$ denote
its marginal over the variables $x_j$,  $j\in\Tree_k(t)$.
In other words $\nu^t$ is the projection of $\nu$ on $\{+1,-1\}^{\Tree_k(t)}$.
For
a measure $\mm\in \M_k$ we let $\mm^t$
denote the measure on the space of measures on
$\{+1,-1\}^{\Tree_k(t)}$ induced by such projection.
\begin{definition}
Consider a sequence of graphs/Ising measures
pairs  $\{(G_n,\mu_n)\}_{n\in\naturals}$ and let
$\prob_n^t(i)$ denote the law of the pair
$(\Ball_i(t),\ux_{\Ball_i(t)})$ when $\ux$ is drawn with distribution $\mu_n$
and $i\in [n]$ is vertex in the graph. Let $U_n$ denote the uniform measure over a random vertex $I \in [n]$.
Let $\prob_n^t = \E_{U_n}(\prob_n^t(I))$ denote the average of $\prob_n^t(I)$.
\begin{enumerate}
\item[A.] The first mode of convergence concerns picking a random vertex $I$ and a random local configuration in the neighbourhood of $I$.  Formally, for $\limMes\in \mathcal{G}_k$ we say that $\{\mu_n\}_{n\in\naturals}$
{\em converges locally on average to} $\limMes$
if for any $t$ and any $\eps > 0$ it holds that
\begin{eqnarray}
\lim_{n\to\infty} d_{\sTV}\left(\prob_n^t, \delta_{\Tree_k(t)} \times \limMes^t \right) = 0.
\end{eqnarray}
\item[B.] A stronger form of convergence involves picking a random vertex $I$ and the associated random local measure $\prob_n^t(I)$ and asking if this distribution of distributions converges. Formally, we say that {\em the local distributions of} $\{\mu_n\}_{n\in\naturals}$
{\em converge locally to} $\mathfrak{m}\in\mathcal{M}_k^{\mathcal{G}}$ if it holds that the law of
$\prob_n^t(I)$ converges weakly to $\delta_{\Tree_k(t)} \times \mathfrak{m}^t$ for all $t$.
\item[C.] If $\mathfrak{m}$ is a point mass on $\limMes\in \mathcal{G}_k$ and if the local distributions of $\{\mu_n\}_{n\in\naturals}$
converge locally to $\mathfrak{m}$
then we say that $\{\mu_n\}_{n\in\naturals}$
{\em converges in probability locally to} $\limMes$.  Equivalently convergence in probability locally to $\limMes$ says that for any $t$ and any $\eps > 0$ it holds that
\begin{eqnarray}
\lim_{n\to\infty} U_n \left(d_{\sTV}(\prob_n^t(I), \delta_{\Tree_k(t)} \times \nu^t \right) > \eps) = 0.
\end{eqnarray}
\end{enumerate}
It is easy to verify that $C \Rightarrow B \Rightarrow A$.
\end{definition}

Similar notions of the convergence has been studied before under the name metastates for Gibbs measures.  Aizenman and Wehr \cite{AizWeh:89}, while investigating the quenched behaviour of lattice random field models, introduced the notion of a metastate which is a probability measures over Gibbs measures as a function of the disorder (the random field).  Here, rather than taking a finite graph and choosing a random vertex they take a fixed random environment in
$\integers^d$, and study the measure over increasing finite volumes.  Rather than prove convergence (which depending on the model may not hold) they take subsequential limits and study the properties of these limiting distributions of Gibbs measures (metastates).  Another, similar notion of convergence to metastates was developed by Newman and Stein \cite{newman:96} where they took the empirical measure over Gibbs measures at over increasing volumes to study spin-glasses.  More references and discussions can be found in \cite{kulske:96}.

In order to state our main result formally, we recall
that an Ising measure on $\Tree_k$ is Gibbs if, for any
integer $t\ge 0$
\begin{eqnarray}
\mu^{\Tree_k(t)|\cTree_k(t)}(\ux_{\Tree_k(t)}|\ux_{\cTree_k(t)})
= \frac{1}{Z_{t,\ux}(\beta)}\exp\left\{
\beta\sum_{(i,j)\in E(\Tree_k(t+1))}x_ix_j\right\}\, ,
\end{eqnarray}
where $Z_{t}(\beta)$ is a normalization function that depends
on the conditioning, namely on $\ux_{\Tree_{k}(t+1)\setminus \Tree_k(t)}$.

It is well known that if $(k-1)\tanh\beta\le 1$, there exist only one
Gibbs measure on a $k$-regular tree while for $(k-1)\tanh\beta>1$
the Gibbs measures form a non-trivial convex set (see e.g.~\cite{georgii88}). Two
of its extreme points, $\nu_{+}$ and $\nu_{-}$ play a special role in
the following. The `plus-boundary conditions' measure $\nu_+$ is
defined as the monotone decreasing limit
(with respect to the natural partial ordering on the
space of configurations $\{+1,-1\}^{\Tree_k}$)  of $\nu_{+}^t$ as $t\to\infty$,
where $\nu_{+}^t$ is the measure on $\ux_{\Tree_k(t)}$ defined by
\begin{eqnarray}
\nu_{+}^t(\ux_{\Tree_k(t)})
= \frac{1}{Z_{+,t}(\beta)}\exp\left\{
\beta\sum_{(i,j)\in E(\Tree_k(t))}x_ix_j\right\}
\prod_{i\in\Tree_{k}(t)\setminus \Tree_{k}(t-1)}\ind(x_i=+1)\, .
\end{eqnarray}
The measure $\nu_-$ is defined analogously, by forcing spins on the boundary
to take value $-1$ instead of $+1$. The two measures are obviously
related through spin reversal. Further it well known
(and easy to prove) that for any Gibbs measure $\nu$ we have
$\nu_-\preceq \nu\preceq \nu_+$
(with $\preceq$ the stochastic ordering induced by the partial ordering
on $\{+1,-1\}$ configurations, see e.g.~\cite{Liggett}).
Our main result may be now stated as follows
\begin{thm}\label{thm:Main}
Let $\{G_n\}_{n\in\naturals}$ be a sequence of $k$-regular graphs that
converge locally to the tree
$\Tree_k$. For $(k-1)\tanh\beta>1$,
 define the sequence $\{\mu_n\}_{n\in\naturals}$,
$\{\mu_{n,+}\}_{n\in \naturals}$ by
\begin{eqnarray}
\mu_{n,+}(\ux) & = & \frac{1}{Z_{n,+}(\beta)}\, \exp\Big\{\beta\sum_{(i,j)\in E_n}x_ix_j
\Big\}\;\ind\Big\{\sum_{i\in V_n}x_i>0\Big\}\, ,\label{eq:IsingModelNoField}\\
\mu_{n}(\ux) &=&\frac{1}{Z_{n}(\beta)}\, \exp\Big\{\beta\sum_{(i,j)\in E_n}x_ix_j
\Big\}.
\end{eqnarray}
Then
\begin{enumerate}
\item[I.]
$\mu_n$ converges locally in probability to $\frac12(\nu_+ + \nu_-)$
\item[II.] If the graphs $\{G_n\}$
are $(1/2,\lambda)$ edge expanders for some $\lambda>0$,
then  $\mu_{n,+}$ converges locally in probability
to the plus-boundary Gibbs measure on the infinite tree
$\nu_+$.
\end{enumerate}
\end{thm}

This characterization has a number of useful consequences.
In particular, `spatial' averages of local functions are
roughly constant under the conditional measure $\mu_{n,+}$.
To be more precise,
for each $i\in V_n$, let
\[
f_{i,n}:\{+1,-1\}^{\Ball_i(\ell)}\to [-1,1],
\]
be a function of its neighborhood $\Ball_i(\ell)$.
\begin{thm}\label{thm:Consequence}
Let $\{G_n\}_{n\in\naturals}$ be a sequence of $k$-regular
$(1/2,\lambda)$ edge expanders, for some $\lambda>0$, that
converge locally to the tree
$\Tree_k$.
For each $n$, let $\{f_{i,n}\}_{i=1}^n$ be a
collection of local functions as above. Then, for any
$\ve>0$
\begin{eqnarray}
\lim_{n\to \infty}\mu_{n,+}\Big\{\Big|
\frac{1}{n}\sum_{i\in V_n}[f_{i,n}(\ux_{\Ball_i(\ell)})- \mu_{n,+}\left(\frac{1}{n}\sum_{i\in V_n}
f_{i,n}(\ux_{\Ball_i(\ell)})) \right)]\Big|\ge \ve \Big\} = 0\, .
\end{eqnarray}
\end{thm}
The proof can be found in Section \ref{sec:Consequence}.
%
%
\subsection{Examples and remarks}

Notice that, for $(k-1)\tanh\beta\le 1$, the set of Ising
Gibbs measures on $\Tree_k$ contains a unique element,
that can be obtained as limit of free boundary measures.
Therefore, the local limits of $\{\mu_n\}_{n\in\naturals}$,
$\{\mu_{n,+}\}_{n\in\naturals}$ coincide trivially
with this unique Gibbs measure.

Therefore, the claim $I$ is proved under the weakest possible, hypothesis,
namely local convergence of the graphs to $\Tree_k$.
An important class of graphs for which Theorem~\ref{thm:Main}
is applicable are random $k$-regular graphs.
These are known to converge locally to $\Tree_k$ \cite{Wormald99}.

The expansion condition (or an analogous `connectedness' condition)
is needed to obtain the
convergence of the conditional measures
$\mu_{n,+}$. For example consider $r$ identical but disjoint graphs on $n/r$ vertices. Then conditioning on the sum of the spins being positive the probability that the sum of spins in a specific component is positive is of order $r^{-1/2}$. Therefore in this case we have:
\[
\mu_{n,+} \to  (1-q) \nu_{+}+ q\nu_{-} \, ,
\]
with $q = 1/2 - O(r^{-1/2})$. A similar construction may be repeated with a small number of edges connecting different components, e.g., when the components are connected in a cyclic fashion.

In order to identify the limit for $\mu_n$ and obtain our results,
there are a number of challenges that need to be overcome. First, while soft compactness arguments imply that subsequential limits exist, such arguments do not imply the existence of a proper limit. Second, recalling that there are uncountably many extremal Gibbs measures for $\Tree_k$, it is remarkable we are able to identify precisely those that appear in the limit. Finally, for conditional measures such as $\mu_{n,+}$ it is not even a priori clear that (subsequential) limits are in fact Gibbs measures.

%
%
\subsection{Proof strategy}
\label{sec:MainProof}

The basic idea of the proof is the following. Look at a ball of radius $t$ around a vertex $i$ in $G_n$.
Since $G_n$ is tree like, the ball is with high probability a tree.
The measure $\mu_n$ restricted to the ball is clearly a Gibbs measure
on a tree of radius $t$. The same is true (although less obvious)
for $\mu_{n,+}$.

In order to characterize the limit of this measure as $n\to\infty$,
\begin{enumerate}
\item[1.]
The probability of agreement between neighboring spins in the ball is asymptotically the same as in the measure $\nu_+$ on the infinite tree.
\item[2.] We further show that $\nu_{+}$ maximizes the probability of agreement between neighboring spins among all Gibbs measures on the tree. These two facts together imply that any local limit must converge to a convex combination of  $\nu_+$ and $\nu_-$.
\item[3.] By symmetry this already implies converges of $\mu_n$ to $\frac{1}{2}(\nu_+ + \nu_-)$. Note that this step does not require expansion, just the local weak convergence of the tree.
\item[4.]
In order to deal with the conditional measure, we use expansion to show that it is unlikely that simultaneously a positive fraction of the vertices have their neighborhood ``in the $+$ state'' and another positive fraction ``in the $-$ state''.
\end{enumerate}
%
%
\section{Proof of the main theorem}
We now proceed with the proof. For each of claims I and II
we break the proof into $3$ steps:
\begin{itemize}
\item[$(i)$] We consider a subsequence of sizes $\{n(m)\}_{m\in\naturals}$
along which $\mu_{n(m)}$ or $\mu_{n(m),+}$ converge locally
in average to a limit $\bar{\nu}$ or $\bnu_+$ (respectively).
\item[$(ii)$]
We prove that any such limit is in fact always the same and is
$\bnu = (1/2) (\nu_+ + \nu_-)$ for $\mu_{n(m)}$ and
(using expansion)
$\bnu_+ = \nu_+$ for $\mu_{n(m),+}$.
As a consequence the sequences themselves converge.
\item[$(iii)$]
Finally we show how is it possible to deduce local convergence
from convergence in average.
\end{itemize}
%
%
\subsection{Subsequential limits}
\label{sec:Subsequential}

The construction of subsequential weak limits is
based on a standard diagonal argument,
for similar results see \cite{AldLyo:07}.
For the sake of simplicity we refer to the measures $\mu_{n,+}$,
and construct the subsequential limit $\bnu_+$,
but the same procedure works for $\mu_n$ with limit $\bnu$.
Let $\Ball_I(t)$ be the ball of radius $t$ centered at a uniformly random
vertex $I$ in $V_n$, and $\ux$ be an Ising configuration with distribution
$\mu_{n,+}$.
If $\prob_n$ denotes the joint distribution of
$(\Ball_I(t),\ux_{\Ball_I(t)})$, we let
\begin{eqnarray}
\mu_{+,n}^t(\ux^*_{\Tree_k(t)})
\equiv \prob_n\big\{(\Ball_I(t),\ux_{\Ball_I(t)})
\simeq (\Tree_k(t),\ux^*_{\Tree_k(t)})\big\}\, .
\end{eqnarray}
Since this is a sequence of measures over a finite state space,
it converges over some subsequence $\{n_t(m)\}_{m\ge 0}$. Further,
since by hypothesis $\prob_n\{\Ball_i(t)
\simeq \Tree_k(t)\}\to 1$, the limits of $\mu_{+,n_t(m)}^t$ and $\mu_{n_t(m)}^t$ are in fact
probability measures. We call the limit $\limMes_+^{t}$.

Fix one of these subsequences $\{n_{t_0}(m)\}_{m\ge 0}$ for $t=t_0$,
leading to the limit $\limMes^{t_0}_+$, and recursively refine it
to  $\{n_{t_0}(m)\}_{m\ge 0}\supseteq \{n_{t_0+1}(m)\}_{m\ge 0}
\supseteq \{n_{t_0+2}(m)\}_{m\ge 0}\supseteq\dots$
leading to limits $\limMes^{t}_+$ for all $t\ge t_0$.
Notice that, for any graph $G_n$, any vertex $i$ and any $t$ we have
\begin{eqnarray}
\mu_{n,+}^t(\ux_{\Ball_i(t)}) = \sum_{\ux_{\Ball_i(t+1)\setminus\Ball_i(t)}}
\mu_{n,+}^{t+1}(\ux_{\Ball_i(t+1)})\, .
\end{eqnarray}
As a consequence, for any $t$, the measures limit $\limMes_+^{(t)}$ measure satisfies
\begin{eqnarray}
\limMes_+^{t}(\ux_{\Tree_k(t)}) = \sum_{\ux_{\Tree_k(t+1)\setminus\Tree_k(t)}}
\limMes_+^{t+1}(\ux_{\Tree_k(t+1)})\, .
\end{eqnarray}
By Kolmogorov extension theorem, there exist  measures
$\limMes_+$ over $\{+1,-1\}^{\Tree_k}$ such that
$\limMes_+^{t}$ are the marginals
of $\limMes_+$ over the variables in the subtree $\Tree_k(t)$.
By taking the diagonal subsequence $n(m)=n_m(m)$ we obtain
the desired subsequence $\{n(m)\}_{m\in\naturals}$ such that
$\mu_{n(m),+}$ converges locally on average to $\limMes_+$.

%
%
\subsection{$\limMes =  \frac{1}{2}(\nu_+ + \nu_-)$}

In this section we carry out our program in the case of the unconditional
measures $\mu_n$.
It is immediate that, since each of the measures
$\mu_n^t$ is a Gibbs measure on $\Tree_k$ (although with
a complicate boundary condition), the limit measure $\limMes$ is
also a Gibbs measure on $\Tree_k$ (i.e. $\limMes \in\G_k$).

For proving convergence of the unconditional measure we need  two lemmas.
The first one establishes that the $+$ (equivalently $-$)
Gibbs measure $\nu_+$ has the correct expected number of
edge disagreements (in physics terms, the correct energy density).
\begin{lemma}\label{lemma:EnergyLimit}
Let $\{G_n\}_{n\in\naturals}$ be a sequence of $k$-regular graphs
converging locally to $\Tree_k$, let $I$ be a uniformly random
vertex in $G_n$, and $J$ be chosen uniformly among its $k$
neighbors.
Then
\begin{eqnarray}
\lim_{n\to\infty}\E_{U_n}\,\mu_{n,+}(x_I \cdot x_J) = \lim_{n\to\infty}
\E_{U_n}\,\mu_n(x_I \cdot x_J) = \nu_+(x_{\root} \cdot x_1) =
\nu_-(x_{\root}\cdot x_1)\, ,
\end{eqnarray}
where $1$ is one of the neighbors of the root in $\Tree_k$,
and $\E_{U_n}$ denotes the expectation over the random edge $(I,J)$
in $G_n$.
\end{lemma}
For the proof of this Lemma we refer to Section \ref{sec:ProofEnergyLimit}.
Notice that $\nu_+$ and $\nu_-$ have the same expectation
of the product $x_{\root}x_1$ by symmetry under inversion
$\{x_i\}\to \{-x_i\}$. The probability that the spins at $\root$ and
$1$ agree is simply $(1+\nu (x_{\root}\cdot x_1))/2$.  The
second Lemma shows that $\nu_+$, $\nu_-$ are uniquely
characterized by this agreement probability among all Ising Gibbs
measures on $\Tree_k$.
\begin{lemma} \label{lemma:Extremal}
Let $\nu$ be a Gibbs measure for the Ising
model on $\Tree_k$. Then
\begin{eqnarray}
\nu(x_{\root}\cdot x_1) \le  \nu_+(x_{\root}\cdot x_{1})=\nu_-(x_{\root}\cdot x_{1})\, ,
\end{eqnarray}
and the inequality is strict unless $\nu$ is a convex combination
of $\nu_+$ and $\nu_-$.
\end{lemma}
The proof of this Lemma can be found in Section
\ref{sec:LemmaExtremal}.
We can now prove the following:
\begin{prop}\label{prop:2}
Let $\{G_n\}_{n\in\naturals}$ be a sequence of $k$-regular graphs that
converge locally to the tree
$\Tree_k$. Then for $(k-1)\tanh\beta>1$, it holds that
$\mu_n$ converges locally in average to $(1/2)(\nu_+ + \nu_-)$.
\end{prop}
\begin{proof}
By Lemma \ref{lemma:EnergyLimit} and weak convergence, we have
$\limMes(x_{\root}\cdot x_1) =
\nu_+(x_{\root}\cdot x_1)$.   By Lemma \ref{lemma:Extremal},
$\bnu = (1-q)\nu_++q\bnu_-$ for some $q\in [0,1]$.
On the other hand $\mu_{n,+}$ is symmetric under spin inversion for each
$n$, and therefore $\bnu$ must be symmetric as well, whence $q=1/2$
\end{proof}

We can now prove the first part of our main result.
\begin{proof}[Proof (Theorem \ref{thm:Main}, part I)]
By a similar construction to the one recalled
in Section \ref{sec:Subsequential}, and compactness of
$\M_k$, we can construct a subsequence $\{n(m)\}_{m\in\naturals}$
such that $\mu_{n(m)}$ converges locally (not only in average)
to a distribution $\mm$ over $\H_k$. By the arguments above,
$\mm$ is in fact a measure over the space Ising Gibbs measures $\G_k$.

We claim that any such subsequential weak limit $\mm$
is in fact a point mass at
$(1/2)(\nu_++\nu_-)$. Since
$\nu\mapsto \nu(x_{\root}\cdot x_1)$ is continuous in the weak topology it follows that
\begin{eqnarray}
\lim_{m\to\infty}\E_{U_n}\mu_{n(m)}(x_I\cdot x_J) = \int \nu(x_{\root}\cdot x_1)
\;\mm(\de\nu)\,.
\end{eqnarray}
By Lemma \ref{lemma:EnergyLimit}, this implies
\begin{eqnarray}
\int \nu(x_{\root}\cdot x_1)\;\de\mm(\nu) = \nu_+(x_{\root}\cdot x_1)\,,
\end{eqnarray}
and  therefore, by Lemma \ref{lemma:Extremal}, $\mm$ is supported
on Ising Gibbs measures $\nu$ that are convex combinations of
$\nu_+$ and $\nu_-$.
Finally, $\mu_{n}$ is almost surely symmetric for any $n$.
Here `symmetric' means that, for any configuration $\ux_{\Ball_i(t)}$,
$\mu_n^t(\ux_{\Ball_i(t)}) =\mu_n^t(-\ux_{\Ball_i(t)})$.
Therefore $\mm$ is supported on Ising Gibbs measures that are
symmetric.

There is  only one Ising Gibbs measure that is a convex combination
of $\nu_+$ and $\nu_-$ and is symmetric, namely $\nu = (1/2)(\nu_++\nu_-)$.
Hence $\mm$ is a point mass on this distribution.
\end{proof}

%
%
\subsection{$\limMes_+ =  \nu_+$}

We now turn to the subsequence of conditional measures
$\{\mu_{n(m),+}\}_{m\in\naturals}$ converging locally in average to
$\limMes_+$.
The goal of this subsection is to show that $\limMes_+$ is equal to $\nu_+$.

For this we repeat the previous proof with two additional ingredients.
First we need to show that $\limMes_+$ is a Gibbs measure on the tree
$\Tree_k$. This requires proof since the conditioning
on $\{\sum_{i\in V_n}x_i>0\}$ implies that the measures $\mu_{n,+}^t$
are not Gibbs measures. The Gibbs property is only recovered in the limit.

Second even after we have established that $\limMes_+$ is a Gibbs measure,
this measure is not symmetric with respect to
spin flip. Therefore the argument above only implies that
$\limMes_+ = (1-q) \nu_+ + q \nu_-$. It remains to show that $q=0$.
This is where the expansion assumption is used.
The first lemma we prove is the following:
\begin{lemma}\label{lemma:Gibbs}
Any subsequential limit $\limMes_+$ constructed as above is
an Ising-Gibbs measure on $\Tree_k$.
\end{lemma}
We defer the proof to Section \ref{sec:ProofGibbs}.  Given Lemma \ref{lemma:Gibbs} the following lemma follows immediately from Lemmas \ref{lemma:EnergyLimit} and \ref{lemma:Extremal}.

\begin{lemma} \label{cor:q}
For any subsequential limit $\limMes_+$ there exists a $q \in [0,1]$ such that
\begin{eqnarray}
\limMes_+ = (1-q)\, \nu_+ + q\, \nu_-\, .\label{eq:Mixture}
\end{eqnarray}
\end{lemma}
\begin{proof}
By Lemma~\ref{lemma:Gibbs} the measure $\limMes_+$ is an Ising Gibbs measure on $\Tree_k$. If it was not a convex combination of $\nu_+$ and $\nu_-$ a contradiction to Lemma~\ref{lemma:Extremal} would be derived.
\end{proof}.

The last step consists of arguing that $q=0$. Given a vertex $i$
(either in a graph $G_n$ of the sequence or of $\Tree_k$), an
integer $\ell\ge 1$ and a random Ising configuration $\ux$,
let
\begin{eqnarray}
\F_{i}(\ell,\delta,\ux) \equiv \ind\Big\{\sum_{j\in\Ball_i(\ell)}x_j\le
-\delta\, |\Ball_i(\ell)|\Big\}\, ,
\end{eqnarray}
where $\delta\in(0,1)$ will be chosen below.  Roughly speaking $\F_i$ indicates which vertices are in the ``$-$ state''.
We will drop reference to $\delta$ and to the configuration $\ux$ when clear
from the context. The following lemmas will be proven in Section
\ref{sec:Markov}.
\begin{lemma}\label{l:Markov}
Let $\{G_n\}$ be a sequence of graphs converging locally to $\Tree_k$,
and, for each $n$, $\ux=\ux(n)$ be a configuration in the
support of $\mu_{n,+}$. Then there exists $n_0$, depending on $\delta$,
$\ell$ and the graph sequence, but not on $\ux$,  such that, for all
$n\ge n_0$,
\begin{eqnarray}
%
\E_{U_n}(\F_I(\ell,\delta,\ux) ) 
\le \frac{1}{1+\delta/2}\, ,
\label{eq:APriori}
\end{eqnarray}
where $\E_{U_n}$ denotes expectation with respect to the uniformly random
vertex $I$ in $V_n$.
\end{lemma}
The following lemma is an immediate consequence of the definition of local weak convergence.
\begin{lemma}\label{l:FlwcLimits}
Consider a uniformly random vertex $I$ in $G_n$, let $J$ be one
of its neighbors (again uniformly random),
and let $\{n(m)\}_{m\in\naturals}$ a subsequence of graph sizes
along which $\mu_{n(m),+}$ converges locally on average to $\limMes_{+}$.
Then we have
\begin{eqnarray}
\lim_{m\to\infty}\E_{U_{n(m)}}\mu_{n(m),+} (\F_I(\ell)) &=&
\limMes_{+}(\F_{\root}(\ell))\, ,\label{eq:LimitF1}\\
\lim_{m\to\infty} \E_{U_{n(m)}}\mu_{n(m),+} (\F_I(\ell)\neq \F_J(\ell )) &=&
\limMes_{+}(\F_{\root}(\ell) \neq \F_1(\ell) )\, ,\label{eq:LimitF2}
\end{eqnarray}
with $\E$ denoting expectation with respect to the law $U_{n(m)}$
of vertices $I$ and $J$, and $1$ one of the neighbors of $\root$.
\end{lemma}
Now the limit quantities can be estimated as follows.
\begin{lemma}\label{lemma:Mixture}
Assume $(k-1)\tanh\beta>1$ and let
$\nu = (1-q)\nu_+ + q\nu_-$ be a mixture of the plus and minus measures
for the Ising model on $\Tree_k$. Then there exist $\delta=\delta(\beta)>0$
such that, letting $\F_i(\ell) = \F_i(\ell,\delta;\ux)$,
\begin{align}
\lim_{\ell\to\infty} \nu(\F_{\root}(\ell) = 1) & =  q\, ,\\
\lim_{\ell\to\infty} \nu(\F_{\root}(\ell) \neq \F_1(\ell) )& =  0\, .
\end{align}
\end{lemma}

We can now prove the following:
\begin{prop}\label{prop:2b}
Let $\{G_n\}_{n\in\naturals}$ be a sequence of $k$-regular graphs that
are $(1/2,\lambda)$ expanders for some $\lambda>0$ and
converge locally to the tree
$\Tree_k$. Then for $(k-1)\tanh\beta>1$, it holds that
$\mu_{n,+}$ converges locally on average to $\nu_+$
\end{prop}
\begin{proof}
Let $n(m)$ be a subsequence along which $\mu_{n,+}$ converges locally on average to some $\limMes_+$.  By Lemma~\ref{cor:q} we can write this in the form $\limMes_+ = (1-q)\, \nu_+ + q\, \nu_-$.  Then by Eqs.~(\ref{eq:LimitF1}), (\ref{eq:LimitF2}), for any $\ve>0$,
there exists $\ell$, such that for large enough $n(m)$,
\begin{eqnarray}
\E \,\mu_{n(m),+}(\F_I(\ell))&\ge &q-\ve\, ,\label{eq:EpsClose1}\\
\E\,\mu_{n(m),+}(\ind\{\F_I(\ell)\neq \F_J(\ell)\})& \le & \ve\, .
\label{eq:EpsClose2}
\end{eqnarray}
On the other hand, since $G_n$ is a $(1/2,\lambda)$ expander,
and using Eq.~(\ref{eq:APriori}), we have
\begin{eqnarray}
\sum_{(i,j)\in E_n}\ind\{\F_i(\ell)\neq \F_j(\ell)\}& \ge &
\lambda\, \min(\sum_{i\in V_n}\F_i(\ell),
\sum_{i\in V_n}(1-\F_i(\ell)))\\
& \ge &\lambda\min(\sum_{i\in V_n}\F_i(\ell),
n\delta/(2+\delta)) \\
&\ge & \frac{\lambda\delta}{2+\delta} \sum_{i\in V_n}\F_i(\ell)\, .
\end{eqnarray}
Recalling (\ref{eq:EpsClose1}), (\ref{eq:EpsClose2}),
taking expectation of both sides with respect to $\mu_{n,+}$
and representing the sums over $E_n$, $V_n$ as expectations, we get
\begin{eqnarray}
\frac{k}{2} \ve \geq
\frac{k}{2}\, \E\mu_{n(m),+}(\ind\{\F_I(\ell)\neq \F_J(\ell)\})
\ge \frac{\lambda\delta}{2+\delta} \E \mu_{n(m),+}(\F_I(\ell))\ \geq
\frac{\lambda\delta}{2+\delta} (q-\ve).
\end{eqnarray}
%
Since $\ve>0$ is arbitrary, we derive a contradiction unless $q=0$.
The proof follows.
\end{proof}

We can now complete the proof of Theorem \ref{thm:Main}.
\begin{proof}[Proof (Theorem \ref{thm:Main}, part II)]
Let $n(m)$ be a subsequence along which the local distributions of $\mu_{n,+}$ converge locally  to some $\mathfrak{m}$ (by
the same compactness arguments used in the previous section,
one always exists).  Now by Proposition \ref{prop:2b} it follows that $\nu_+=\int_{\mathcal{G}_k} \nu \ \mm(\de\nu)$ which implies that $\mathfrak{m}$ is a point measure on $\nu_+$ since it is extremal.  This implies local convergence in probability to $\nu_+$, which completes the proof.
\end{proof}
%
%
\section{Proofs of Lemmas}\label{sec:Proofs}

%
%
\subsection{Proof of Lemma \ref{lemma:Gibbs}}
\label{sec:ProofGibbs}

We start from a very general remark, which is implicit in~\cite{DobrushinTirozzi}
holding
for a general Markov random field on a graph $G=(V,E)$
\begin{eqnarray}
\mu(\ux) = \frac{1}{Z}\prod_{(i,j)\in E}\psi_{i,j}(x_i,x_j)
\end{eqnarray}
where $\ux=\{x_i\}_{i\in V}\in\cX^V$ for a finite spin alphabet $\cX$,
and $\psi_{ij}:\cX\times \cX\to\reals$ is a collection of potentials.
Recall that a subset $S$ of the vertices of $G$
is  an independent set  if, for any $i,j\in S$, $(i,j)\not\in E$.
\begin{lemma}\label{lemma:CLT}
Assume $0<\psi_{\rm min}\le \psi_{ij}(x_i,x_j)\le \psi_{\rm max}$,
let $k$ be the maximum degree of $G$, and $I(G)$ the maximum size
of an independent set of $G$. Then there exists a constant
$C = C(k,\psi_{\rm max}/\psi_{\rm min})>0$ such that,
for any $x\in\cX$ and any $\ell\in\naturals$,
\begin{eqnarray}
\mu\Big(\sum_{i\in V}\ind_{x_i=x}=\ell\Big)\le \frac{C}{\sqrt{I(G)}}\, .
\end{eqnarray}
\end{lemma}
\begin{proof}
Let $S$ be a maximum size independent set and $\oS = V\setminus S$
its complement. Further, let
$Y_U\equiv \sum_{i\in U}\ind_{x_i=x}$ for $U\subseteq V$.
Conditioning on $\ux_{\oS} = \{x_i:\, i\in\oS\}$
\begin{eqnarray}
\mu\Big(\sum_{i\in V}\ind_{x_i=x}=\ell\Big) =
\E_{\mu}\Big\{\mu\Big(Y_S=\ell-Y_{\oS}|\ux_{\oS}\Big)\Big\}\, .
\end{eqnarray}
Conditional on $\ux_{\oS}$, the variables $\{x_i\}_{i\in S}$
are independent with $\delta\le \mu(x_i=x|\ux_{\oS})\le 1-\delta$
for some $\delta>0$ depending on $k$ and $\psi_{\rm max}/\psi_{\rm min}$.
As a consequence $Y_S$ is the sum of $|S| = I(G)$ independent
Bernoulli random variables with expectation bounded away from $0$ and 1.
By the Berry-Esseen Theorem
\begin{eqnarray}
\mu\Big(Y_S=\ell-Y_{\oS}|\ux_{\oS}\Big)\le \frac{C}{\sqrt{I(G)}},
\end{eqnarray}
which implies the thesis.
\end{proof}

\begin{proof}(Lemma \ref{lemma:Gibbs})
Recall that for $\Tree_k$, the infinite rooted $k$-regular tree,
we denote by $\Tree_k(t)$ the subtree induced by nodes
with distance at most $t$ from the root $\root$.
Also, denote $\Tree_k(t,t_+)=\Tree_k(t_+)\setminus \Tree_k(t)$, the subgraph
induced by nodes $i$ with distance $t+1\le d(i,\root)\le t_+$.
Let $\limMes_+$ denote a subsequential limit of the measures $\mu_{n,+}$ constructed as in Section
\ref{sec:MainProof}. For any $t\ge 1$ and $t_+>t$ we will prove that
the conditional distribution of $\ux_{\Tree_k(t)}$ given
$\ux_{\Tree_k(t,t_+)}$ is given by
(here and below we adopt the  convention of writing
$p(x|y) \normeq f(x,y)$ for a conditional distribution $p$,
whenever  $p(x|y) = f(x,y)/\sum_{x'}f(x',y)$):
\begin{eqnarray}
\limMes_{+}^{\Tree_k(t)|\Tree_k(t,t+)}(\ux_{\Tree_k(t)}|\ux_{\Tree_k(t,t_+)})
\normeq \exp\left\{
\beta\sum_{(i,j)\in E(\Tree_k(t+1))}x_ix_j\right\}\, .
\label{eq:ClaimConditional}
\end{eqnarray}
This establishes the DLR conditions and implies that $\limMes_{+}$ is a Gibbs measure as required.

In analogy with the notation introduced above
(and recalling that $\Ball_i(t)$ is the ball of radius $t$
around vertex $i$ in $G_n$), we let $\Ball_i(t,t_+)=\Ball_i(t_+)\setminus
\Ball_i(t)$ be the
subgraph induced by vertices $j$ such that $t+1\le d(i,j)\le t_+$.
Also $\Ell_i(t)$ will be the set of edges in $\Ball_i(t)$,
and $\cEll_i(t) = E_n\setminus\cEll_i(t)$.
The marginal distribution of $\ux_{\Ball_i(t_+)}$ under $\mu_{n,+}$
is given by
\begin{eqnarray}
\mu_{n,+}^{t_+}(\ux_{\Ball_i(t_+)})
& \normeq & F_{\Ball_i(t_+)}(\ux_{\Ball_i(t_+)})\; Z_{\Ball_i(t_+)}(\ux_{\Ball_i(t_+)})
\, \\
F_{\Ball_i(t_+)}(\ux_{\Ball_i(t_+)})
&\equiv &\exp\Big\{\beta\!\!\sum_{(l,j)\in\Ell_i(t_+)}\!\! x_lx_j\Big\}\;
\, ,\\
Z_{\Ball_i(t_+)}(\ux_{\Ball_i(t_+)}) &\equiv &
\sum_{\ux_{V_n\setminus\Ball_i(t_+)}}\exp\Big\{\beta\!\!
\sum_{(l,j)\in\cEll_i(t_+)}x_l x_j
\Big\}\; \ind\Big(\sum_{j\in\cBall_i(t_+)}x_j> -\sum_{j\in\Ball_i(t_+)}x_j\Big)\, .
\end{eqnarray}
We, therefore, have the following expression for the conditional
distribution of $\ux_{\Ball_i(t)}$, given $\ux_{\Ball_i(t,t_+)}$:
\begin{eqnarray}
\mu_{n,+}^{\Ball_i(t_+)|\Ball_i(t,t_+)}(\ux_{\Ball_i(t_+)}|\ux_{\Ball_i(t,t_+)})
= \frac{ F_{\Ball_i(t_+)}(\ux_{\Ball_i(t_+)})\; Z_{\Ball_i(t_+)}(\ux_{\Ball_i(t_+)})}
{\sum_{\ux_{\Ball_i(t)}}F_{\Ball_i(t_+)}(\ux_{\Ball_i(t_+)})\;
Z_{\Ball_i(t_+)}(\ux_{\Ball_i(t_+)})}\, .\label{eq:ConditionalProbability}
\end{eqnarray}
On the other hand we have
$Z^{-}_{\Ball_i(t_+)}(\ux_{\Ball_i(t_+)})\le
Z_{\Ball_i(t_+)}(\ux_{\Ball_i(t_+)})\le Z^{+}_{\Ball_i(t_+)}(\ux_{\Ball_i(t_+)})$
where we define
\begin{eqnarray}
Z_{\Ball_i(t_+)}^{\pm}(\ux_{\Ball_i(t_+)}) &\equiv &
\sum_{\ux_{V_n\setminus\Ball_i(t_+)}}\exp\Big\{\beta\!\!
\sum_{(l,j)\in\cEll_i(t_+)}x_ix_j
\Big\}\; \ind\Big(\sum_{j\in\cBall_i(t_+)}x_j> \mp |\Ball_i(t_+)|\Big)\, .
\end{eqnarray}
Notice that $Z_{\Ball_i(t_+)}^{\pm}(\ux_{\Ball_i(t_+)})$
depend on $\ux_{\Ball_i(t_+)}$ only through  $\ux_{\Ball_i(t,t_+)}$.
Using the expression (\ref{eq:ConditionalProbability})
for the conditional probability (and dropping subscripts on $\mu$
to lighten the notation), we have
\begin{align}
\mu_{n,+}(\ux_{\Ball_i(t_+)}|\ux_{\Ball_i(t,t_+)})
&\le \mu^{*}(\ux_{\Ball_i(t_+)}|\ux_{\Ball_i(t,t_+)})
\;   \max_{\ux\in \{+1,-1\}^{\Tree_k(t,t_+)} }\frac{Z_{\Ball_i(t_+)}^{+}(\ux)}
{Z_{\Ball_i(t_+)}^{-}(\ux)},\\
\mu_{n,+}(\ux_{\Ball_i(t_+)}|\ux_{\Ball_i(t,t_+)})
&\ge \mu^{*}(\ux_{\Ball_i(t_+)}|\ux_{\Ball_i(t,t_+)}) \min_{\ux\in \{+1,-1\}^{\Tree_k(t,t_+)} } \frac{Z_{\Ball_i(t_+)}^{-}(\ux)}
{Z_{\Ball_i(t_+)}^{+}(\ux)}\, ,
\end{align}
with
\begin{eqnarray}
\mu^{*}(\ux_{\Ball_i(t_+)}|\ux_{\Ball_i(t,t_+)})
\normeq \exp\Big\{
\beta\sum_{(l,j)\in \Ell_i(t+1)}x_lx_j\Big\}\, .
\end{eqnarray}
The claim (\ref{eq:ClaimConditional}) thus follows
from the fact that $\Ball_i(t_+)\simeq \Tree_k(t_+)$ with probability
going to $1$ as $n\to\infty$, if we can show
that
\begin{equation}\label{e:conditioningLimit}
\frac{Z_{\Ball_i(t_+)}^{-}(\ux)}
{Z_{\Ball_i(t_+)}^{+}(\ux)} \to 1
\end{equation}
for all $\ux\in \{+1,-1\}^{\Tree_k(t,t_+)}$ as $n\to\infty$.

Let $\hmu$ denote the Ising measure on $\ux_{\cBall_i(t_+)}$
with boundary conditions $\ux_{\Ball_i(t_+)}$
\begin{eqnarray}
\hmu(\ux_{\cBall_i(t_+)})= \frac1{\hat Z(\ux_{\Ball_i(t_+)})} \exp\Big\{\beta\!\!
\sum_{(l,j)\in\cEll_i(t_+)}x_ix_j
\Big\}\, .
\end{eqnarray}
Now
\[
1- \frac{Z_{\Ball_i(t_+)}^{-}(\ux)}
{Z_{\Ball_i(t_+)}^{+}(\ux)}  = \frac{\hmu\Big(\sum_{j\in\cBall_i(t_+)}x_j> - |\Ball_i(t_+)|\Big) - \hmu\Big(\sum_{j\in\cBall_i(t_+)}x_j>  |\Ball_i(t_+)|\Big)}{\hmu\Big(\sum_{j\in\cBall_i(t_+)}x_j> - |\Ball_i(t_+)|\Big)}\, .
\]
Observe that by the Gibbs construction of $\mu$ for any $\ux_{\cBall_i(t_+)}$, we have that
\begin{align*}
\hmu(\ux_{\cBall_i(t_+)}) &\geq \exp(-2\beta k |\Ball_i(t^+)|)\mu_n(\ux_{\cBall_i(t_+)})
\end{align*}
as this is the maximum affect that conditioning on a set of size $|\Ball_i(t^+)|$ can have on the measure $\mu$.  By symmetry of the measure $\mu_n$ with respect to the sign of $\ux$,
\begin{align}\label{e:positiveProb}
\hmu\Big(\sum_{j\in\cBall_i(t_+)}x_j> - |\Ball_i(t_+)|\Big) &\geq \hmu\Big(\sum_{j\in\cBall_i(t_+)}x_j \geq 0\Big)\nonumber\\
&\geq \exp(-2\beta k |\Ball_i(t^+)|)\mu_n\Big(\sum_{j\in\cBall_i(t_+)}x_j \geq 0\Big)\nonumber\\
&\geq \frac12  \exp(-2\beta k |\Ball_i(t^+)|).
\end{align}
Now applying Lemma \ref{lemma:CLT} to the measure $\hmu$ we have that
\begin{align}\label{e:balancedProb}
\hmu\Big(\sum_{j\in\cBall_i(t_+)}x_j> - |\Ball_i(t_+)|\Big) - \hmu\Big(\sum_{j\in\cBall_i(t_+)}x_j>  |\Ball_i(t_+)|\Big) &= \hmu\Big(\big|\sum_{j\in\cBall_i(t_+)}x_j\big| \leq |\Ball_i(t_+)|\Big) \nonumber\\
&\leq \frac{2C|\Ball_i(t_+)|}{\sqrt{n - |\Ball_i(t_+)|}} \to 0
\end{align}
for some $C=C(k,\beta)$ as $n\to \infty$.
Combining equations \eqref{e:balancedProb} and \eqref{e:positiveProb} we establish equation \eqref{e:conditioningLimit} which completes the proof.
\end{proof}
%
%
\subsection{Proof of Lemma \ref{lemma:EnergyLimit}}
\label{sec:ProofEnergyLimit}

For the convenience of the reader, we restate the main result of \cite{DemboMontanariIsing}
in the case of $k$-regular graphs, with no magnetic field $B$.
This provides an asymptotic
estimate of the partition function
\begin{eqnarray}
Z_n(\beta) = \sum_{\ux}\exp\Big\{\beta\sum_{(i,j)\in E}x_ix_j
+\sum_{i\in V}x_i\Big\}\, .
\end{eqnarray}
\begin{thm}
Let $\{G_n\}_{n\in\naturals}$ be a sequence of graphs
that converges locally to the $k$-regular tree $\Tree_k$.
For $\beta>0$, let $h$ be the largest solution
of
\begin{eqnarray}
h = (k-1) \tanh[\tanh(\beta) \tanh (h)]\, .\label{eq:TreeFixedPoint}
\end{eqnarray}
Then $\lim_{n\to\infty}\frac{1}{n}\log Z_n = \phi(\beta)$, where
\begin{eqnarray}\label{eqn:phi}
\phi(\beta) &\equiv& \frac{k}{2}\, \log\cosh(\beta) -
\frac{k}{2}\,\log\{1+\tanh(\beta) \tanh(h)^2\}
\nonumber
\\
&+& \log\Big\{ [1+\tanh(\beta) \tanh(h)]^k +
[1 -\tanh(\beta) \tanh(h) ]^k \Big\}\, ,
\end{eqnarray}
\end{thm}

For the proof of Lemma \ref{lemma:EnergyLimit} we start by noticing that, by
symmetry under change of sign of the $x_i$'s,
we have $\mu_{n,+}(x_i\cdot x_j)=  \mu_n(x_i \cdot x_j)$.
Simple calculus yields
\begin{eqnarray}
\frac{1}{n}\,\frac{\partial\phantom{\beta}}{\partial\beta}
\log Z_n(\beta) = \frac{1}{n}\sum_{(i,j)\in\E_n}\mu_n (x_i\cdot x_j) =
\frac{k}{2}\, \E \mu_n(x_I \cdot x_J) \, ,
\end{eqnarray}
where the expectation $\E$ is taken with respect to $I$
uniformly random vertex, and $J$ one of its neighbors taken
uniformly at random.

On the other hand, differentiating Eq.~(\ref{eqn:phi}) with respect to
$\beta$, and using the fixed point condition (\ref{eq:TreeFixedPoint}),
we get after some algebraic manipulations
\begin{eqnarray}
\frac{\partial\phantom{\beta}}{\partial\beta}\phi(\beta)
= \frac{k}{2}\,
\frac{\tanh\beta+(\tanh h)^2}{1+\tanh\beta(\tanh h)^2}
= \frac{k}{2} \nu_+(x_{\root}\cdot x_1)\, .
\end{eqnarray}
The last identification comes from the fact that the
joint distribution of $x_{\root}$ and $x_1$ on a $k$-regular
tree under the plus-boundary Gibbs measure is
$\nu_+(x_{\root},x_1) \propto \exp\{\beta x_{\root}x_1+hx_{\root}+hx_1\}$
(see \cite{DemboMontanariIsing}).

Further $\beta\mapsto \frac{1}{n}\log Z_n(\beta)$ is convex
because its second derivative is proportional to the variance
of $\sum_{(i,j)}x_ix_j$ with respect to the measure $\mu_n$.
Therefore, its derivative $(k/2)\E \mu_n(x_i \cdot x_j)$ converges
to $(k/2)\nu_+(x_{\root}\cdot x_1)$ for a dense subset of values of
$\beta$. Since the limit $\beta\mapsto  \nu_+(x_{\root}\cdot x_1)$
is continuous, convergence takes place for every $\beta$.
%
%
\subsection{Proof of Lemma \ref{lemma:Extremal}}
\label{sec:LemmaExtremal}

Recalling that $\Tree_k$ denotes the infinite $k$-regular tree rooted
at $\root$ let $\Tree^{\root}$ and $\Tree^1$ be the subtrees
obtained by removing the edge $(\root,1)$ where $1$ is a neighbor of $\root$.
It is sufficient to prove the claim when $\nu$ is an extremal Gibbs
measure on $\Tree_k$ since of course we may decompose any Gibbs
measure into a mixture of extremal measures.
For $i\in\{\root,1\}$ define
\[
m_i^{\nu} = \lim_{\ell\to\infty}\E_{\Tree^i}(x_i\mid
\ux_{\Ball_i^c(\ell)\cap \Tree^i})
\]
where $\E_{\Tree^i}$ denotes expectation with respect to the
Ising model on the tree $\Tree_i$ and the boundary condition
$\ux_{\Ball_i(\ell)\cap \Tree^i}$ is chosen according to $\nu$.
The limit exists by the
Backward Martingale Convergence Theorem. Further it is a constant
almost surely, because it is measurable with respect to
the tail $\sigma$-field, and
$\nu$ is extremal.

By the monotonicity of the Ising model if $\nu \preceq \nu'$,
then $m_i^\nu\leq m_i^{\nu'}$.  Furthermore
\begin{equation}\label{e:rootField}
\nu(x_{\root})= \frac{m^{\nu}_{\root}+\tanh(\beta)m_1^{\nu}}
{1+\tanh(\beta)m^{\nu}_{\root} m_{1}^{\nu}}\, .
\end{equation}
Now if $\nu\neq\nu^+$ then $\nu(x_{\root}=1)<\nu^+(x_{\root}=1)$.  Under
the plus measure $m_{\root}^{\nu_+}=m_{1}^{\nu_+}=m^+$ which by the
monotonicity of the system is the maximal such value.  Since
the right hand side of Eq.~\eqref{e:rootField} is increasing in $m_{\root},m_1$ it follows that $m_{\root}^{\nu}=m_{1}^{\nu}=m^+$ if and only if
$\nu=\nu^+$.

An easy tree calculation shows that the expectation of $x_{\root}\cdot x_1$
is
\begin{align*}
\nu(x_{\root}\cdot x_1) = \frac{\tanh(\beta)+m^{\nu}_{\root}m^{\nu}_1}
{1+\tanh(\beta)m^{\nu}_{\root}m^{\nu}_1}\, .
\end{align*}
which is strictly increasing in $m^{\nu}_{\root}$ when $m^{\nu}_1>0$.
By symmetry it is also strictly increasing in $m^{\nu}_1$ when $m^{\nu}_{\root}
>0$.  Hence amongst measures $\nu$ with $m^{\nu}_{\root}\geq 0$,
the expectation $\nu(x_{\root}\cdot x_1)$ is uniquely maximized when
$m_{\root}^{\nu}=m_{1}^{\nu}=m^+$, that is when $\nu=\nu^+$.  Similarly amongst measures $\nu$ with $m^\nu_{\root}\leq 0$ the agreement probability is uniquely maximized by $\nu_-$, which completes the proof.

%
%
\subsection{Proof of Lemma \ref{l:Markov}}
\label{sec:Markov}

Observe first by the local weak convergence of the graphs $\{G_n\}$ that
all but $o(n)$ vertices appear in $|\Ball_i(\ell)|$ balls
$\Ball_i(\ell)$.  Hence given a configuration $\ux$ with
$\sum_ix_i\ge 0$, we have
\begin{eqnarray}
\sum_{i\in V_n}\left(\frac{1}{|\Ball_i(\ell)|}\sum_{j\in\Ball_i(\ell)}x_j\right)
\ge -o(n)\, .
\end{eqnarray}
By Markov's inequality (applied to the uniform choice of $i\in V_n$)
we have
\begin{eqnarray}
\frac{1}{n}\sum_{i\in V_n}\F_i(\ell)\le \frac{1}{1+\delta}+o_n(1)
\le \frac{1}{1+\delta/2}\, ,
\end{eqnarray}
where the second inequality holds for all $n$ large enough

%
%
\subsection{Proof of Lemma \ref{lemma:Mixture}}

Setting $\rho=\nu_+(x_{\root})$ note that by invariance of
$\nu_+$ under graph homomorphisms of $\Tree_k$, we have
\[
\nu_+\bigg( \sum_{j\in\Ball_i(\ell)}x_j \bigg) = \rho\left|\Ball_i(\ell)\right|.
\]
Moreover, under $\nu_+$, along any path of vertices in $\Tree_k$ the states are distributed as a 2-state homogenous Markov chain and hence
\[
\nu_+(x_j \cdot x_{j'})-
\nu_+(x_j)\nu_+(x_j') = A\, b^{d(j,j')}
\]
where $d(j,j')$ is the graph distance between vertices $i$ and $j$,
and $b\in (0,1)$ is a constant depending on $\beta$.

This in particular implies that
\[
\mathrm{Var}_{\nu_+}\bigg( \sum_{j\in\Ball_i(\ell)}x_j \bigg) = o\bigg(\Big|\Ball_i(\ell)\Big|^2\bigg)\, ,
\]
and therefore, using Chebychev inequality,
$\frac1{\Ball_i(\ell)}\sum_{j\in\Ball_i(\ell)}x_j$ converges in probability
to $\rho$ as $\ell\to\infty$.  Similarly under the measure $\nu_-$ we have that $\frac1{\Ball_i(\ell)}\sum_{j\in\Ball_i(\ell)}x_j$ converges in probability to $-\rho$.  Now taking $0<\delta<\rho$ we have that
\begin{align*}
\lim_{\ell\to\infty} \nu_+(\F_{\root}(\ell) = 1) & =  0\, ,\\
\lim_{\ell\to\infty} \nu_-(\F_{\root}(\ell) = 1 )& =  1\, .
\end{align*}
Therefore, for $\nu = (1-q)\nu_+ + q\nu_-$,
we have $\nu_+(\F_{\root}(\ell) = 1)\to q$.

Moreover, by translation invariance
\begin{align*}
\nu_+(\F_{\root}(\ell)\neq \F_{1}(\ell)) =
2\nu_+(\F_{\root}(\ell)=1, \F_{1}(\ell)=0) \le
2\nu_+(\F_{\root}(\ell)=1)\to 0\, .
\end{align*}
By applying the same argument to $\nu_-$, we deduce
that the probability that $\F_{\root}(\ell)$ and $\F_1(\ell)$ differ goes to
$0$ under any mixture of $\nu_+$ and $\nu_-$.
Since $\nu$ is a mixture of $\nu_+$ and $\nu_-$ this completes the lemma.
%
%
\section{Proof of Theorem \ref{thm:Consequence}}
\label{sec:Consequence}

To simplify notation we will write $f_i$ or $f_i(\ux)$ for $f_{i,n}(\ux_{\Ball_i(\ell)})$.
We will prove that, denoting by $\Var_{n,+}$, $\Cov_{n,+}$
variance and covariance under $\mu_{n,+}$,
\begin{eqnarray*}
\lim_{n\to\infty}\Var_{n,+}\Big(\frac{1}{n}\sum_{i\in V_n}f_i(\ux_{\Ball_i(\ell)})\Big)
= \lim_{n\to\infty}
\E_{U_n}\Cov_{n,+}(f_I(\ux_{\Ball_I(\ell)}),f_L(\ux_{\Ball_L(\ell)})) = 0\, .
\end{eqnarray*}
Here $\E_{U_{n}}$ denotes
expectation with respect to two independent and uniformly
random vertices $I,L$ in $V_n$. The thesis then follows by Chebyshev
inequality.

Since the $f_i$'s are bounded, we have for $r > \ell$,
\begin{eqnarray*}
\E_{U_n}\Cov_{n,+}(f_I,f_L) \le \prob_{U_n}(d(I,L) \leq 2 r)
+\E_{U_n}\Big\{\Cov_{n,+}(f_I,f_L) ; d(I,L) > 2 r \Big\}\, .
\end{eqnarray*}
Since $\{G_n\}_{n\in\naturals}$ are $k$-regular,
the probability $d(I,L) \leq 2 r$ vanishes as $n\to\infty$.
It therefore suffices to show that
\[
\lim_{r \to \infty} \lim_{n \to \infty} \E_{U_n}\Big\{\Cov_{n,+}(f_U,f_V) ; d(U,V) > 2 r \Big\}\ = 0\, .
\]
Define
\begin{eqnarray*}
\hf_{i}^+(r)(\ux) = \E_{n,+}\{f(\ux_{\Ball_i(\ell)})|
x_{V_n\setminus\Ball_i(r)}\}\,,
\end{eqnarray*}
the conditional expectation being taken with respect to $\mu_{n,+}$.
Then we have for all $i,j$ that
\begin{eqnarray*}
%
\ind(d(i,j) > 2 r)\, \Cov_{n,+}(f_i,f_j) = \ind(d(i,j) > 2 r)\, \Cov_{n,+}(\hf_{i}^+(r),f_j) \le
\sqrt{\Var_{n,+}(\hf_{i}^+(r))}\,
\end{eqnarray*}
and therefore
\begin{eqnarray}
\lim_{r \to \infty} \lim_{n \to \infty} \E_{U_n}\Big\{\Cov_{n,+}(f_I,f_L) ; d(I,L) > 2 r \Big\}\ &\leq&
\lim_{r \to \infty} \lim_{n \to \infty} \E_{U_n} \sqrt{\Var_{n,+}(\hf_{I}^+(r))} \\
& \leq &
\lim_{r \to \infty} \lim_{n \to \infty} \sqrt{\E_{U_n} \Var_{n,+}(\hf_{U}^+(r))}\, .
\end{eqnarray}
Define the modified function
\begin{eqnarray}
\hf_{i}(r)(\ux) = \E_{n}\{f(\ux_{\Ball_i(\ell)})|
x_{V_n\setminus\Ball_i(r)}\}\,,
\end{eqnarray}
where the expectation is taken with respect to the measure
$\mu_n$. Since the latter is a Gibbs measure $\hf_{i}(r)$
depends on $\ux$ only through the variables $x_j$, $j\in\Ball_i(r)
\setminus\Ball_i(r-1)$. Further $\hf_{i}^+(r)$ and $\hf_{i}(r)$ differ only if
$|\sum_{j\in V_n\setminus\Ball_i(r)} x_j|\le |\Ball_i(r)|$.
Therefore
\begin{eqnarray*}
\Var_{n,+}(\hf_{I}^+(r)) & \le &2\Var_{n,+}(\hf_{I}(r))
+2\Var_{n,+}(\hf_{I}^+(r)-\hf_{I}(r))\\
&\le &2\Var_{n,+}(\hf_{I}(r))
+8\mu_{n,+}\Big(\Big|\sum_{j\in V_n\setminus\Ball_I(r)} x_j\Big|\le
|\Ball_I(r)|\Big)
\, .
\end{eqnarray*}
The last term vanishes as $n\to\infty$ by Lemma \ref{lemma:CLT}.

We are therefore left with the task of showing that
$\lim_{r \to \infty} \lim_{n\to \infty}\E_{U_n}\Var_{n,+}(\hf_{I}(r)) = 0$.
For a function $f : \{-1,1\}^{\Tree_k(\ell)} \to [-1,1]$, let
\[
\barf(r)(\ux) = \E_{\nu_+}\{f(\ux_{\Tree_k(\ell)})|\ux_{\Tree_k\setminus \Tree_k(r)}\}.
\]
For all functions whose domain is not  $\{-1,1\}^{\Tree_k(\ell)}$
we let $\barf(r) = 0$ by convention.
Also, with an abuse of notation, we define
$\barf_{i}(r) = \bar g(r)$ for $g = \hf_i$.
Since  $\hf_{I}(r)$ depends on $\ux$ only through
$\ux_{\Ball_I(r)}$, we obtain by Theorem~\ref{thm:Main} for every $\ve > 0$ that
\[
\lim_{n \to \infty} \E_{U_n}|\Var_{n,+}(\hf_{I}(r)) - \Var_{\nu_+}(\barf_I(r))| \leq 2 \ve +
\lim_{n\to\infty} U_n \left(d_{\sTV}\left(\prob_n^t(I), \delta_{\Tree_k(t)} \times \nu^t_+ \right) > \ve \right) = 2 \ve\, ,
\]
and therefore
\[
\lim_{r \to \infty} \lim_{n \to \infty} \E_{U_n} \Var_{n,+}(\hf_{I}(r)) \leq
\lim_{r \to \infty} \sup \left\{ \Var_{\nu_+}(\barf(r)) \,\,\mid \,\,\ f : \{-1,1\}^{\Tree_k(\ell)} \to [-1,1] \right\}.
\]
By extremality of $\nu_+$, for each $f : \{-1,1\}^{\Tree_k(\ell)} \to [-1,1]$,
$\barf(r)$ converges to an almost sure constant as $r\to\infty$ and
since $f$ is bounded,
$\lim_{r\to\infty}\Var_{\nu_+}(\barf(r)) = 0$.
For each $r$, the map $f \to \barf(r)$ is a contraction in $L^2$ and therefore
the map $f \to \sqrt{\Var_{\nu_+}(\barf(r))}$ is a Lipchitz map with constant $1$.
Since the set of functions $f : \{-1,1\}^{\Tree_k(\ell)} \to [-1,1]$ is compact in $L_2$ and for each
$f$ we have $\lim_{r\to\infty}\Var_{\nu_+}(\barf(r)) = 0$ we conclude that
\[
\lim_{r \to \infty} \sup \left\{ \Var_{\nu_+}(\barf(r)) \,\,\ | \,\,\ f : \{-1,1\}^{\Tree_k(\ell)} \to [-1,1] \right\} = 0,
\]
as needed.
%
%
\section*{Acknowledgements}

A.M. was partially supported by
 by a Terman fellowship, the NSF CAREER award CCF-0743978
and the NSF grant DMS-0806211.
E.M. was partially supported  by the
NSF CAREER award grant DMS-0548249, by DOD ONR
grant (N0014-07-1-05-06), by ISF grant 1300/08 and by EU grant
PIRG04-GA-2008-239317.

Part of this work was carried out while two of the authors
(A.M. and E.M.) were visiting Microsoft Research.
%
%
\bibliographystyle{amsalpha}

\end{document}